\documentclass[12pt]{amsart}
\usepackage{amssymb,amsmath,amsthm,amsfonts,times,hyperref,mathrsfs,multirow}
\newtheorem{theorem}{Theorem}[section]
\newtheorem{lemma}[theorem]{Lemma}
\newtheorem{cor}[theorem]{Corollary}
\newtheorem{conj}[theorem]{Conjecture}

\theoremstyle{definition}

\newtheorem{example}[theorem]{Example}

\theoremstyle{remark}
\newtheorem{remark}[theorem]{Remark}

\numberwithin{equation}{section}

\newcommand\leg[2]{\genfrac{(}{)}{}{}{#1}{#2}} 
\newcommand\pdx{\frac{\partial}{\partial x} }
\newcommand\rvl[2]{\left.{#1}\right|_{#2}}  
\newcommand\dtp{\left(\frac{d}{dt}\right)}  
\newcommand\qdqp{\left(q\frac{d}{dq}\right)}  
\allowdisplaybreaks

\begin{document}
\newcommand\mylabel[1]{\label{#1}}
\newcommand{\beqs}{\begin{equation*}}
\newcommand{\eeqs}{\end{equation*}}
\newcommand{\beq}{\begin{equation}}
\newcommand{\eeq}{\end{equation}}
\newcommand\eqn[1]{(\ref{eq:#1})}
\newcommand\thm[1]{\ref{thm:#1}}
\newcommand\lem[1]{\ref{lem:#1}}
\newcommand\propo[1]{\ref{propo:#1}}
\newcommand\corol[1]{\ref{cor:#1}}
\newcommand\sect[1]{\ref{sec:#1}}

\title[$r$-Fishburn numbers]
{Congruences and relations for $r$-Fishburn numbers}

\author{F. G. Garvan}
\address{Department of Mathematics, University of Florida, Gainesville,
FL 32611-8105}
\email{fgarvan@ufl.edu}
\thanks{}


\date{\today}                   



\begin{abstract}
Recently Andrews and Sellers proved some amazing congruences
for the Fishburn numbers. We extend their results to a more general
sequence of numbers. As a result we prove a new congruence mod $23$
for the Fishburn numbers and prove their conjectured mod $5$ congruence
for a related sequence. We also extend and prove some unpublished conjectures 
of Garthwaite and Rhoades.
\end{abstract}

\maketitle

\section{Introduction}
\mylabel{sec:intro}

The Fishburn numbers $\xi(n)$ \cite{A022493} are defined by
the formal power series
\beq
\sum_{n=0}^\infty \xi(n) q^n = F(1 - q),
\eeq
where  
\beq
F(q) := \sum_{n=0}^\infty (q;q)_n,
\mylabel{eq:Fdef}
\eeq
and  
$$
(a ;q)_n = (1 - a) (1 -aq) \cdots (1 - aq^{n-1}).
$$
Zagier \cite{Za} showed that $\xi(n)$ is the
number of linearized chord diagrams of degree $n$, and  also 
the number of nonisomorphic interval orders on $n$ unlabeled points.
Andrews and Sellers \cite{An-Se} proved some amazing
congruences for the Fishburn numbers. For example, for
all $n\ge0$,
\begin{align}
  \xi(5n+3) &\equiv \xi(5n+4)\equiv 0 \pmod{5},  \mylabel{eq:mod5congs} \\
  \xi(7n+6) &\equiv 0 \pmod{7},  \mylabel{eq:mod7congs} \\
  \xi(11n+8) &\equiv \xi(11n+9)\equiv \xi(11n+10)\equiv 0 \pmod{11}, 
  \mylabel{eq:mod11congs}\\
  \xi(17n+16) &\equiv 0 \pmod{17}, \text{\ \ and} \mylabel{eq:mod17congs} \\
  \xi(19n+17) &\equiv \xi(19n+18)\equiv 0 \pmod{19}. \mylabel{eq:mod19congs}
\end{align}

In fact, they prove that there are analogous congruences for
all primes $p$ that are quadratic nonresidues mod $23$. For $p$
prime they define 
\begin{equation}
\mylabel{eq:Sdef}    
S(p) = \left\{ j \,:\,  0\leq j\leq p-1 
       \text{\ such that\ } \tfrac{1}{2}n(3n-1) \equiv j \pmod{p} 
       \text{\ for some } n   \right\}
\end{equation}
and 
\begin{equation}
\mylabel{eq:Tdef}
T(p) = \left\{ k \,:\, 0\leq k\leq p-1 
       \text{\ such that\ } 
       k \text{\ is larger than every element of $S(p)$}    \right\}.
\end{equation}
We state their main result.
\begin{theorem}[Andrews and Sellers\cite{An-Se}]
\mylabel{thm:ASmainthm}
If $p$ is a prime and $i\in T(p)$ (as defined in \eqn{Tdef}), 
then for all $n\geq 0,$ 
$$\xi(pn+i) \equiv 0 \pmod{p}.$$
\end{theorem}
\begin{remark}
Congruences \eqn{mod5congs}--\eqn{mod19congs} are the cases $p=5$, $7$,
$11$, $17$ and $19$ of Theorem \thm{ASmainthm}.  Andrews and Sellers
proved that $T(p)$ is nonempty whenever $p$ is a quadratic nonresidue 
mod $23$.
\end{remark}

Recently Garthwaite and Rhoades \cite{Ga-Rh} observed congruence pairs and 
triples 
such as
\begin{align}
\xi(5n+2) - 2\,\xi(5n+1) &\equiv0 \pmod{5}, \mylabel{eq:xcong5}\\
\xi(11n+7) - 3\,\xi(11n+4) + 2\,\xi(11n+3) &\equiv 0 \pmod{11}.
\mylabel{eq:xcong11}
\end{align}
In Theorem \thm{mainthm} and Corollary \corol{congrels} below we prove 
congruences relations mod $p$ exist for all primes $\ge5$.

We extend the Andrews and Sellers result to what we call $r$-Fishburn
numbers $\xi_r(n)$ and which we define by
the formal power series
\beq
\sum_{n=0}^\infty \xi_r(n) q^n = F((1 - q)^r),
\eeq
where $r$ is any nonzero integer. The case $r=1$ corresponds to
the ordinary Fishburn numbers. Even in the case $r=1$ we are able
to augment the set $T(p)$.  
For $p \ge 5$ prime, $r$  relatively
prime
to $p$ and nonzero and $0\le s \le p-1$ we define
\begin{align}
\mylabel{eq:Srdef}    
S^*(p,r,s) &= \left\{ j \,:\,  0\leq j\leq p-1 
       \text{\ such that\ } \tfrac{1}{2}r n(3n-1) \equiv j-s \pmod{p} 
       \text{\ for some }  n \right.\\
       &\qquad \left. \text{and\ }24(j-s)\not\equiv -r\pmod{p}\right\}
\nonumber
\end{align}
and 
\begin{equation}
\mylabel{eq:Trdef}
T^*(p,r,s) = \left\{ k \,:\, 0\leq k\leq p-1 
       \text{\ such that\ } 
       k \text{\ is larger than every element of $S^*(p,r,s)$}    \right\}.
\end{equation}

We state our main
\begin{theorem}
\mylabel{thm:mainthm}
Suppose $p\ge5$ is prime, $r$ is a nonzero integer relatively
prime to $p$ and $0 \le s \le p-1$. If $m\in T^*(p,r,s)$ (as defined in \eqn{Trdef}), 
then for all $n\geq 0,$ 
$$
\sum_{j=0}^s \binom{s}{j} (-1)^j \xi_r(pn+m -j) \equiv 0 \pmod{p}.
$$
\end{theorem}
\begin{remark}
We make some remarks.
\begin{enumerate}
\item[(i)]
The $(r,s)=(1,0)$ case of the theorem is slightly stronger than the
Andrews-Sellers result. We observe that although $22\in S(23)$,
$22\not\in S^*(23,1,0)$ and we find $T^*(23,1,0)=\{18,19,20,21,22\}$ so that
\beq
\xi(23n+18) \equiv \xi(23n+19)\equiv \xi(23n+20) \equiv
\xi(23n+21) \equiv \xi(23n+22)\equiv 0 \pmod{23}. 
\mylabel{eq:newmod23cong}
\eeq
This is a congruence that Andrews and Sellers missed.
\item[(ii)]
We find that $S^*(5,-1,0)=\{0,3\}$ and $T^*(5,-1,0)=\{4\}$ so that
\beq
\xi_{-1}(5n+4) \equiv 0 \pmod{5}.
\mylabel{eq:a-s-conj}
\eeq
This congruence was conjectured by Andrews and Sellers \cite{An-Se}.
When $(r,s)=(-1,0)$ the Theorem only gives a congruence in the case $p=5$.
This is because $1$ is a pentagonal number so that when $p>5$ we have $p-1\in S^*(p,-1)$
and $T^*(p,-1,0)$ is empty. Using the facts that
$$
T^*(5,-1,2)=\{3,4\},\quad T^*(5,-1,3)=\{4\},
$$                         
we find that
\beq
\xi_{-1}(5n+3) \equiv 3\,\xi_{-1}(5n+2) \equiv 2\,\xi_{-1}(5n+1) \pmod{5}.
\mylabel{eq:xaconj5}
\eeq
\end{enumerate}
\end{remark}
\begin{example}
$$
S^*(43,-1,2) = \{0, 1, 2, 5, 7, 10, 13, 14, 16, 18, 19, 23, 29, 30, 31, 33, 37, 38, 39, 40, 41\},
$$
so that
$$
T^*(43,-1,2) = \{42\},
$$
and
$$
\xi_{-1}(43n+42) - 2\,\xi_{-1}(43n+41) + \xi_{-1}(43n+40) \equiv 0 \pmod{43},
$$
for all $n\ge 0$.
\end{example}

We highlight the $s=0$ case of the theorem.

\begin{cor}
\mylabel{cor:maincong}
Suppose $p\ge5$ is prime and $r$ is a nonzero integer relatively
prime to $p$. If $m\in T^*(p,r,0)$ (as defined in \eqn{Trdef}), 
then for all $n\geq 0,$ 
$$
\xi_r(pn+m) \equiv 0 \pmod{p}.
$$
\end{cor}

\begin{cor}
\mylabel{cor:congrels}
Suppose $p\ge5$ is prime, and $r$ is a nonzero integer relatively
prime to $p$. Then there are at least $\frac{1}{2}(p+1)$ linearly independent
congruence relations mod $p$ of the form
$$
\sum_{j=0}^{p-1} \alpha_j\, \xi_r(pn + j) \equiv 0 \pmod{p},
$$
where $n$ is any nonnnegative integer and $\vec{\alpha}\in\mathbb{F}_p^{p}$.
\end{cor}
\begin{remark}
In Section \sect{proofmainthm} we prove Corollary \corol{congrels}
by showing that the relations
\beqs
\sum_{j=0}^s \binom{s}{j} (-1)^j \xi_r(pn+ p-1-j) \equiv 0 \pmod{p},
\eeqs
where the Legendre symbol $\leg{-24(1+s)\overline{r}+1}{p}=-1$ or $0$,
form a set of $\frac{1}{2}(p+1)$ linearly independent
congruence relations mod $p$. 
Here $r\overline{r}\equiv1\pmod{p}$.
This also means that $s$ can never
equal $p-1$. 
\end{remark}
\begin{example}
When $p=7$ and $r=1$ there are $4$ relations mod $7$:
\begin{align*}
                          \xi(7 n + 6) &\equiv0\pmod{7},\\  
 \xi(7 n + 6) - 2\, \xi(7 n + 5) + \xi(7 n + 4) &\equiv0\pmod{7},\\
 \xi(7 n + 6) - 3\, \xi(7 n + 5) + 3\, \xi(7 n + 4) - \xi(7 n + 3) &\equiv0\pmod{7},\\
 \xi(7 n + 6) - 4\, \xi(7 n + 5) + 6\, \xi(7 n + 4) - 4\, \xi(7 n + 3) 
+ \xi(7 n + 2) &\equiv0\pmod{7},
\end{align*}
which can be rewritten as
\begin{align*}
                          \xi(7 n + 6) &\equiv0\pmod{7},\\  
                  \xi(7 n + 5) + 5\, \xi(7 n + 2) &\equiv0\pmod{7},\\  
                  \xi(7 n + 4) + 3\, \xi(7 n + 2) &\equiv0\pmod{7},\\ 
                   \xi(7 n + 3) + \xi(7 n + 2) &\equiv0\pmod{7}.            
\end{align*}
\end{example}
\begin{conj}
Suppose $p\ge5$ is prime, and $r$ is a nonzero integer relatively
prime to $p$. Then there are exactly $\frac{1}{2}(p+1)$ linearly independent
congruence relations mod $p$ of the form
$$
\sum_{j=0}^{p-1} \alpha_j\, \xi_r(pn + j) \equiv 0 \pmod{p},
$$
where $n$ is any nonnnegative integer and $\vec{\alpha}\in\mathbb{F}_p^{p}$.
\end{conj}

Following Andrews and Sellers \cite{An-Se}, we define
\beq
F(q,N) = \sum_{n=0}^N (q;q)_n,
\mylabel{eq:FqNdef}
\eeq
and the $p$-dissection
\beq
F(q,N) = \sum_{i=0}^{p-1} q^i A_p(N,i,q^p).
\mylabel{eq:pdissF}
\eeq
We consider the coefficients of the polynomials
\beq
A_p(pn-1,i,1-q) = \sum_{k\ge0} \alpha(p,n,i,k) q^k.
\mylabel{eq:alphacoeffs}
\eeq
The Andrews-Sellers Theorem \thm{ASmainthm} depends crucially on
\begin{lemma}[Andrews and Sellers \cite{An-Se}]
\mylabel{lem:ASlemma}
If $i\not\in S(p)$, then
$$
\alpha(p,n,i,k) = 0,
$$
for $0 \le k \le n-1$.  
\end{lemma}
We consider the analog of this result when $24i\equiv -1 \pmod{p}$.
For $p\ge 5$ prime we define $\bar\xi_p(n)$ by the formal power
series
\beq
\sum_{n=0}^\infty \bar\xi_p(n) q^n = (1-q)^{\lfloor \frac{p}{24} \rfloor}
F( (1-q)^p ),
\mylabel{eq:barxip}
\eeq
where $F(q)$ is defined in \eqn{Fdef}. Observe that when $5 \le p \le 23$,
$$
\bar\xi_p(n) = \xi_p(n),
$$
for $n\ge0$. We find the following new relation for
Fishburn numbers.
\begin{theorem}
\mylabel{thm:newxithm}
Suppose $p\ge 5$ is prime and $24i_0\equiv -1 \pmod{p}$ where
$1 \le i_0 \le p-1$. Then
$$
\alpha(p,n,i_0,k) =  p \, \leg{12}{p}\,\bar\xi_p(k),
$$
for $0 \le k \le n-1$.
Here $\leg{\cdot}{\cdot}$ is the Kronecker symbol.
\end{theorem}
Our main Theorem \thm{mainthm}  will follow from Lemma \lem{ASlemma}
and Theorem \thm{newxithm}
in a straightforward manner.

\section{Preliminary results}
\mylabel{sec:prelim}
\begin{lemma}
\mylabel{lem:alphacoeffs}
Let $p$ be prime and suppose $0\le j\le p-1$ and $0\le k \le M-1 \le N-1$.
Then
$$
\alpha(p,N,j,k)=\alpha(p,M,j,k),
$$
where $\alpha(p,n,i,k)$ is defined in \eqn{alphacoeffs}.
\begin{proof}
Let $\zeta=\exp(2\pi i/p)$. Then from \eqn{pdissF} we have
$$
A_p(N,j,q) = \frac{1}{p} \sum_{k=0}^{p-1} \zeta^{-jk} q^{-\frac{j}{p}}
F\left(\zeta^k q^{\frac{1}{p}},N\right).
$$
Next we suppose that $n\ge pM$. Then
\begin{align*}
\left( \zeta^k (1-q)^{\frac{1}{p}}; \zeta^k (1-q)^{\frac{1}{p}} \right)_n 
&= \prod_{j=1}^n \left(1 - \left(\zeta^k (1-q)^{\frac{1}{p}}\right)^j\right)
\\
&=\prod_{j=1}^{\lfloor n/p \rfloor} ( 1 - (1-q)^j )
\prod_{\substack{j=1 \\ j\not\equiv0\pmod{p}}}^{n}               
\left(1 - \left(\zeta^k (1-q)^{\frac{1}{p}}\right)^j\right)
\\
&=\prod_{j=1}^M ( j q + O(q^2) )
\prod_{\substack{j=1 \\ j\not\equiv0\pmod{p}}}^{n}               
\left(1 - \left(\zeta^k (1-q)^{\frac{1}{p}}\right)^j\right)
\\
& =\quad O\left(q^M\right).
\end{align*}
Thus
\begin{align*}
A_p(pN-1,j,1-q) &= \frac{1}{p} \sum_{k=0}^{p-1} \zeta^{-jk} q^{-\frac{j}{p}}
\sum_{n=0}^{pN-1} 
\left( \zeta^k (1-q)^{\frac{1}{p}}; \zeta^k (1-q)^{\frac{1}{p}} \right)_n 
\\
& = A_p(pM-1,j,1-q) + O(q^M).
\end{align*}
The result follows.
\end{proof}
\end{lemma}
Andrews and Sellers define a Stirling like array of numbers
$C(n,i,j,p)$ for $n\ge0$, $0\le i \le p-1$, and $0\le j \le n$,
which are defined by the recursion
\beq
C(n+1,i,j,p) = (i + jp) C(n,i,j,p) + p C(n,i,j-1,p),
\mylabel{eq:Crec}
\eeq
and the initial value
\beq
C(0,i,0,p) = 1.
\mylabel{eq:Cinit}
\eeq
It is understood that if either of the conditions
$n\ge 0$ or $0 \le j \le n$ are not satisfied, then
$C(n,i,j,p)=0$. We note that
$$
C(n,i,0,p) = i^n.
$$
We need a generalization of the signless Stirling numbers of the first
kind. We define the numbers $s_1(n,j,m)$ for $0\le j \le n$ by
\beq
\sum_{j=0}^n s_1(n,j,m) x^j = (x - m)(x - m + 1) \cdots (x - m + n -1).
\mylabel{eq:s1def}
\eeq
We note case $m=0$ correspond to the signless Stirling numbers of the
first kind $s_1(n,j)$. We define
\beq
f(x,n,k,m) = (-1)^n \sum_{j=k}^n \binom{j}{k} s_1(n,j,m) x^j,
\mylabel{eq:fdef}
\eeq
for $0 \le k \le n$, otherwise define $f(x,n,k,m)=0$.
\begin{lemma}
\mylabel{lem:frec}
For $0 \le k \le n+1$ we have
\beq
f(x,n+1,k,m) = -( (x + n - m) f(x,n,k,m) + x f(x,n,k-1,m) ).
\mylabel{eq:frec}
\eeq
\end{lemma}
\begin{proof}
First we observe that
$$                
f(x,n,k,m) = \frac{x^k}{k!} \left(\frac{\partial}{\partial x}\right)^k
f(x,n,0,m),
$$               
where
$$
f(x,n,0,m)= (-1)^n (x - m)(x - m + 1) \cdots (x - m + n -1) 
\quad\text{(by \eqn{s1def})},
$$
for $k\ge0$. Now we let
$$
\tilde{f}(x,n,k,m) = 
\left(\frac{\partial}{\partial x}\right)^k
(-1)^n (x - m)(x - m + 1) \cdots (x - m + n -1),
$$
for $0 \le k \le n$, otherwise define $\tilde{f}(x,n,k,m)=0$.
We show that 
\beq
\tilde{f}(x,n+1,k,m) = -\left( (x -m+n) \tilde{f}(x,n,k,m) 
                       + k \tilde{f}(x,n,k-1,m) \right)
\mylabel{eq:tildefrec}
\eeq
where $0 \le k \le n+1$. Since    
$$
\tilde{f}(x,n+1,0,m) = -(x-m+n) \tilde{f}(x,n,0,m),
$$
we have
$$
\frac{\partial}{\partial x} \tilde{f}(x,n+1,0,m)
= -(x-m+n) \frac{\partial}{\partial x} \tilde{f}(x,n,0,m) 
  - \tilde{f}(x,n,0,m),
$$
so that
$$  
\tilde{f}(x,n+1,1,m) = -\left( (x -m+n) \tilde{f}(x,n,1,m) 
                       + \tilde{f}(x,n,0,m) \right)
$$
and \eqn{tildefrec} holds for $k=1$. We assume \eqn{tildefrec}
holds for $k \le K$. 
\begin{align*}
\tilde{f}(x,n+1,K+1,m) &= \pdx \tilde{f}(x,n+1,K,m) \\
&= -\pdx\left[ 
(x -m+n) \tilde{f}(x,n,K,m) 
                       + K \tilde{f}(x,n,k-1,m) \right]\\
&= -\left[
(x -m+n) \tilde{f}(x,n,K+1,m)  + \tilde{f}(x,n,K,m)
                       + K \tilde{f}(x,n,K,m) \right]\\
&= -\left[
(x -m+n) \tilde{f}(x,n,K+1,m)  + (K+1) \tilde{f}(x,n,K,m)
                       \right],  
\end{align*}
and \eqn{tildefrec} holds for $k=K+1$. Hence \eqn{tildefrec}
holds for all $k$ by induction. Since
$$
f(x,n,k,m) = \frac{x^k}{k!} \tilde{f}(x,n,k,m),
$$
the result \eqn{frec} follows easily.
\end{proof}

\begin{theorem}
\mylabel{thm:A1id}
Let $i_0 = (p^2 - 1)z - mp$. Suppose that
\beq
\sum_{\ell=0}^n C(n,i_0,\ell,p) A_1(\ell,m) = (-1)^n z^n 
\sum_{k=0}^n \binom{n}{k}\,X(k),
\mylabel{eq:A1eqns}
\eeq
for $n\ge0$. Then
\beq
A_1(n,m) = (-1)^n \sum_{k=0}^n \sum_{j=k}^n 
                  \binom{j}{k} s_1(n,j,m) p^{j-2k} X(k) z^j,
\mylabel{eq:A1id}
\eeq
for $n\ge0$.
\end{theorem}
\begin{proof}
Since for $n\ge 0$ \eqn{A1eqns} forms a triangular system of equations
in the unknowns $A_1(\ell,m)$ (for fixed $m$) and each diagonal coefficient
$$
C(n,i_0,n,p) = p^n \ne 0,
$$
it suffices to show that $A_1(n,m)$ given by \eqn{A1id} satisfies
\eqn{A1eqns}. By considering the coefficient of $X(L)$ it suffices to
show that
\beq
\sum_{\ell=0}^n C(n,i_0,\ell,p) f(pz,\ell,L,m) = G(n,L),
\mylabel{eq:XLid}
\eeq
where
$$
G(n,L) = (-1)^n \binom{n}{L} z^n p^{2L},
$$
for $0 \le L \le n$ and $m\ge 0$. We proceed by induction on $n$.
The result is clearly true for $n=0$. We assume \eqn{XLid} holds for $n=N$.
Now
\begin{align*}
& \sum_{\ell=0}^{N+1} C(N+1,i_0,\ell,p) f(pz,\ell,L,m) \\
&= \sum_{\ell=0}^{N+1} 
\left( ( (p^2-1)z -mp + \ell p) C(N,i_0,\ell,p) + p C(N,i_0,\ell-1,p)\right)
f(pz,\ell, L, m) \quad\text{(by \eqn{Crec})}\\
&= \sum_{\ell=0}^{N} 
( (p^2-1)z -mp + \ell p) C(N,i_0,\ell,p) f(pz,\ell, L, m)\\
&\quad + \sum_{\ell=0}^{N} p C(N,i_0,\ell,p) f(pz,\ell+1,L,m)
\\
&= \sum_{\ell=0}^{N} 
( (p^2-1)z -mp + \ell p) C(N,i_0,\ell,p) f(pz,\ell, L, m)\\
&\quad -\sum_{\ell=0}^{N} p C(N,i_0,\ell,p) \bigg(
  (pz+\ell-m) f(pz,\ell,L,m) + pz f(pz,\ell,L-1,m) \bigg)\\
&=-z \sum_{\ell=0}^N  C(N,i_0,\ell,p) f(pz,\ell, L, m)
  -p^2z \sum_{\ell=0}^N  C(N,i_0,\ell,p) f(pz,\ell, L-1, m)\\
& \hskip 3in \text{(by \eqn{frec})}\\
&=-zG(N,L) -p^2zG(N,L-1)\\
&= (-1)^{N+1}\bigg( \binom{N}{L} z^{N+1}p^{2L} + \binom{N}{L-1}z^{N+1}p^{2L}
\bigg)\\
& =(-1)^{N+1} \binom{N+1}{L} z^{N+1} p^{2L} = G(N+1,L),
\end{align*}
and \eqn{XLid} holds for $n=N+1$, thus completing our induction proof.
\end{proof}

We will also need some results of Zagier \cite{Za} on the
Fishburn numbers. We need the following formal power series 
identity \cite[Eqn.(4),p.946]{Za}
\beq
e^{t/24} \sum_{n=0}^\infty (1 - e^t)\cdots(1 - e^{nt}) =
\sum_{n=0}^\infty \frac{T_n}{n!} \left(\frac{-t}{24}\right)^n,
\mylabel{eq:ZAG1}
\eeq
where $T_n$ are the Glaisher $T$-numbers \cite{A002439} and which are given
explicitly by
\beq
T_n = 6 \frac{(-144)^n}{n+1}\left[
B_{2n+2}\left(\frac{1}{12}\right) -
B_{2n+2}\left(\frac{5}{12}\right)\right],
\mylabel{eq:Tn}
\eeq
where $B_n(x)$ denotes the $n$-th Bernoulli polynomial. 
We remark
that letting $t=\log(1-q)$ in \eqn{ZAG1} and using \eqn{s1njkgen} below
we find that
\beq
\xi(n) = \sum_{m=0}^n \sum_{k=0}^m (-1)^{n+k} \binom{-1/24}{n-m} 
\frac{s_1(m,k)}{ m! 24^k } T_k,
\eeq
which is useful for calculation.
Zagier \cite{Za} also determined the behaviour of $F(q)$ when
$q$ is near a root of unity. In particular,
if $\zeta=\zeta_p$ is a $p$-th root of unity and $N=12p$ then
\beq
e^{t/24} F(\zeta e^t) 
= \sum_{n=0}^\infty \frac{c_n(\zeta)}{n!} \left(\frac{-t}{24}\right)^n,
\mylabel{eq:ZAG2}
\eeq
where
\beq
c_n(\zeta) = \frac{(-1)^n N^{2n+1}}{2n+2}
\sum_{m=1}^{N/2} \chi(m) \zeta^{\frac{1}{24}(m^2-1)}
B_{2n+2}\left(\frac{m}{N}\right),
\mylabel{eq:ZAG3}
\eeq
and where $\chi$ is the character mod $12$ given by
\beq
\chi(n) = \leg{12}{n} = 
\begin{cases}
1 & \mbox{if $n\equiv\pm1\pmod{12}$,}\\
-1 & \mbox{if $n\equiv\pm5\pmod{12}$,}\\
0 & \mbox{otherwise.}
\end{cases}
\mylabel{eq:chi12}
\eeq
This character occurs in the statement of Theorem \thm{newxithm}.

\section{Proof of Theorem \thm{newxithm}}
\mylabel{sec:proofnewxithm}
In this section we assume $p>3$ is prime, $N=12p$ and $\zeta$ is any $p$-th 
root of unity. Following
\cite{Za} we define the sequence $(b_n(\zeta))$ formally by
\beq
F(\zeta e^t) = \sum_{n=0}^\infty \frac{b_n(\zeta)}{n!} t^n.
\mylabel{eq:ZAG4}
\eeq
From \eqn{ZAG1}, \eqn{ZAG2}, \eqn{ZAG3} we have for $n\ge0$
\beq
b_n(\zeta) = \frac{(-1)^n}{24^n} \sum_{j=0}^n\binom{n}{j}
\sum_{i=0}^{p-1} \gamma(j,i)\zeta^i,
\mylabel{eq:nbz}
\eeq
where
\beq
\gamma(j,i) = \frac{(-1)^j N^{2j+1}}{2j+2}
\sum_{\substack{m=1 \\ (m^2-1)/24 \equiv i\pmod{p}}}^{N/2}
\chi(m) B_{2j+2}\left(\frac{m}{N}\right).
\mylabel{eq:gamji}
\eeq
As in \cite{An-Se} we have for $n\ge0$
$$
b_n(\zeta) = \rvl{ \dtp^n F(\zeta e^t)}{t=0}
           = \rvl{ \qdqp^n F(q)}{q=\zeta}
           = \rvl{ \qdqp^n F(q,m)}{q=\zeta}
$$
for $m\ge (n+1)p-1$. Proceeding as in the proof of \cite[lemma 2.5]{An-Se}
we have from \eqn{ZAG3}, \eqn{nbz} and Lemma \lem{alphacoeffs} that
\begin{align*}
b_n(\zeta) &= \rvl{ \qdqp^n F(q,(n+1)p-1)}{q=\zeta} \\ 
&=\sum_{j=0}^n\sum_{i=0}^{p-1} C(n,i,j,p) \zeta^i A_p^{(j)}(p(n+1)-1,i,1)\\
&=\sum_{j=0}^n\sum_{i=0}^{p-1} C(n,i,j,p) \zeta^i A_p^{(j)}(p(j+1)-1,i,1)\\
&=\frac{(-1)^n}{24^n}
\sum_{j=0}^n\binom{n}{j}
\sum_{i=0}^{p-1} \gamma(j,i)\zeta^i.
\end{align*}
Since this identity holds for all $p$-th roots of unity $\zeta$ (including
$\zeta=1$) and all the coefficients involved are all rational numbers
we may equate coefficients of $\zeta^i$ on both sides to obtain
\beq
\sum_{j=0}^n C(n,i,j,p) A_p^{(j)}(p(j+1)-1,i,1)
=
\frac{(-1)^n}{24^n}
\sum_{j=0}^n\binom{n}{j} \gamma(j,i),
\mylabel{eq:Apjid}
\eeq
for $0 \le i \le p-1$. Now we let $i_0$  be the least nonnegative integer
satisfying $24i_0\equiv-1\pmod{p}$. 
We find that
\beq
i_0 = \frac{p^2-1}{24} - \left\lfloor\frac{p}{24}\right\rfloor p.
\mylabel{eq:i0}
\eeq 
We now calculate $\gamma(j,i_0)$.
We see that     
$$
\frac{m^2-1}{24} \equiv i_0\pmod{p} \text{\ if and only if\ }
m\equiv 0\pmod{p}.
$$
In the sum \eqn{gamji} (with ($N=12p$) we only consider the terms
with $m=p$, $5p$ to find that
\beq
\gamma(j,i_0) =  \chi(p) \frac{(-1)^j 12^{2j+1}p^{2j+1}}{2j+2}
\left(
B_{2j+2}\left(\frac{1}{12}\right)
-
B_{2j+2}\left(\frac{5}{12}\right)\right).
\mylabel{eq:gamji0}
\eeq
We apply Theorem \thm{A1id} with $z=\frac{1}{24}$ and 
$m=\left\lfloor\frac{p}{24}\right\rfloor$ to equation \eqn{Apjid}
with $i=i_0$ 
to obtain 
\begin{align*}
A_p^{(n)}(p(n+1)-1,i_0,1)
&=(-1)^n \sum_{k=0}^n \sum_{j=k}^n
\binom{j}{k}
s_1(n,j,\lfloor{p/24}\rfloor) \frac{p^{j+1}}{24^j} \\
& \qquad 
\chi(p) \frac{(-1)^k 12^{2k+1}}{2k+2}\left(
B_{2k+2}\left(\frac{1}{12}\right) -
B_{2k+2}\left(\frac{5}{12}\right)\right)\\
&=(-1)^n \chi(p) \sum_{k=0}^n \sum_{j=k}^n
\binom{j}{k}
s_1(n,j,\lfloor{p/24}\rfloor) \frac{p^{j+1}}{24^j}T_k,
\end{align*}
by \eqn{Tn}.
From \eqn{alphacoeffs} we see that
$$
\alpha(p,n+1,i_0,n) = \frac{(-1)^n}{n!} A_p^{(n)}(p(n+1)-1,i_0,1).
$$
Hence
\begin{align}
\sum_{n=0}^\infty \alpha(p,n+1,i_0,n) x^n 
&= p \chi(p) \sum_{n=0}^\infty \frac{x^n}{n!} 
\sum_{k=0}^n \sum_{j=k}^n
\binom{j}{k}
s_1(n,j,\lfloor{p/24}\rfloor) \frac{p^{j}}{24^j}T_k, 
\nonumber\\
&=
p \chi(p) \sum_{j=0}^\infty\sum_{k=0}^j
\left(\sum_{n=j}^\infty 
s_1(n,j,\lfloor{p/24}\rfloor) \frac{x^n}{n!} \right)
 \frac{p^j}{24^j} \binom{j}{k} T_k  
\mylabel{eq:apni0id}
\end{align}
It is well-known that the signless Stirling numbers of the
first kind have generating function
\beq
G(x, u) 
= \exp \left( -u \log {1-x} \right) 
= \left({1-x} \right)^{-u} = 
\sum_{n=0}^\infty \sum_{j=0}^n s_1(n,j) u^j \, \frac{x^n}{n!}. 
\mylabel{eq:Gxu}
\eeq
Our generalized signless Stirling numbers have generating
function
\beq
G_k(x, u) 
= (1-x)^k \exp \left( -u \log {1-x} \right) 
= \left({1-x} \right)^{k-u} = 
\sum_{n=0}^\infty \sum_{j=0}^n s_1(n,j,k) u^j \, \frac{x^n}{n!}. 
\mylabel{eq:Gkxu}
\eeq
Thus
\beq
\sum_{n=j}^\infty s_1(n,j,k) \frac{x^n}{n!} 
= (1 - x)^k \frac{\left( - \log(1-x) \right)^j}{j!}.
\mylabel{eq:s1njkgen}
\eeq
Hence from \eqn{s1njkgen} and \eqn{apni0id} we have
\begin{align}
\sum_{n=0}^\infty \alpha(p,n+1,i_0,n) q^n 
&= p \chi(p) 
 (1 - q)^{\lfloor{p/24}\rfloor} 
\sum_{j=0}^\infty
\left(
\sum_{k=0}^j
 \binom{j}{k} T_k  
\right)
 \frac{(-p\log(1-q))^j}{24^j j!}
\mylabel{eq:apni0id2}
\end{align}
From Zagier's result \eqn{ZAG1} we have
\beq
F(\exp(t)) = \sum_{n=0}^\infty \left(\sum_{k=0}^n \binom{n}{k} T_k\right) 
\frac{(-t)^n}{24^n n!}.
\mylabel{eq:Fexp}
\eeq
Thus using this result in \eqn{apni0id2} we have
\begin{align}
\sum_{n=0}^\infty \alpha(p,n+1,i_0,n) q^n 
&= p \chi(p) 
 (1 - q)^{\lfloor{p/24}\rfloor} 
 F(\exp(p\log(1-q))  \nonumber\\
&= p \chi(p) 
 (1 - q)^{\lfloor{p/24}\rfloor} F((1-q)^p).
\mylabel{eq:apni0id3}
\end{align}
Theorem \thm{newxithm} follows from \eqn{barxip} and \eqn{apni0id3} 
since
$$
\alpha(p,n,i_0,k) = \alpha(p,k+1,i_0,k),
$$
for $0 \le k \le n-1$ by Lemma \lem{alphacoeffs}.

\section{Proof of Theorem \thm{mainthm} and Corollary \corol{congrels}}
\mylabel{sec:proofmainthm}
We assume $p\ge 5$ is prime and $24i_0\equiv -1 \pmod{p}$ where
$1 \le i_0 \le p-1$. 
For two formal power series 
$$
A = \sum_{n=0}^\infty a_n q^n, \qquad
B = \sum_{n=0}^\infty b_n q^n \in \mathbb{Z}[[q]],
$$
we write
$$
A \equiv B \pmod{p}\quad\text{if}\quad a_n\equiv b_n \pmod{p},
$$
for all $n\ge0$. If $p$ is prime and $m$ is a nonnegative integer
$$
m = j p + r,\quad\text{where $0\le r < p$ and $j$, $r\in\mathbb{N}$},
$$
then
\beq
(1 - q)^m = (1- q)^r (1-q)^{pj} \equiv (1 - q)^r (1 - q^p)^j \pmod{p}.
\mylabel{eq:binomcong}
\eeq
Now 
suppose $r$ is a nonzero integer relatively
prime to $p$ and $0 \le s \le p-1$.  We consider two cases.
\subsection*{Case I} $r>0$. We proceed as in \cite[Section 3]{An-Se}.
We note that
\beq
\sum_{n=0}^\infty \xi(n) q^n = F(1-q,N) + O(q^{N+1}).
\mylabel{eq:xiF}
\eeq
Now from \eqn{pdissF} we have
\begin{align*}
F(q, pn - 1) &= \sum_{i=0}^{p-1} q^i A_p(pn-1,i,q^p) \\
&= \sum_{i\in S(p) \setminus \{i_0\}} q^i A_p(pn-1,i,q^p) + q^{i_0} A_p(pn-1,i_0,q^p) \\
&\qquad \qquad +
\sum_{i\not\in S(p)} q^i A_p(pn-1,i,q^p).
\end{align*}
Hence
\begin{align*}
F((1-q)^r, pn - 1) 
&= \sum_{i\in S(p) \setminus \{i_0\}} (1-q)^{ri} A_p(pn-1,i,(1-q)^{rp}) \\
&\qquad\qquad + (1-q)^{ri_0} A_p(pn-1,i_0,(1-q)^{rp}) \\
&\qquad \qquad +
\sum_{i\not\in S(p)} (1-q)^{ri} A_p(pn-1,i,(1-q)^{rp}).
\end{align*}
For $i\not\in S(p)$ we have
$$
A_p(pn-1,i,q) = \sum_{k\ge0} \alpha(p,n,i,k) (1 - q)^k,
$$
and
\begin{align*}
A_p(pn-1,i,(1-q)^{rp}) &= \sum_{k\ge0} \alpha(p,n,i,k) (1 - (1- q)^{rp})^k\\
&\equiv \sum_{k\ge0} \alpha(p,n,i,k) (1 - (1- q^p)^{r})^k \pmod{p}\\
&\equiv O(q^{pn}) \pmod{p},
\end{align*}
by Lemma \lem{ASlemma}. In a similar fashion we have
$$
A_p(pn-1,i_0,(1-q)^{rp}) 
\equiv O(q^{pn}) \pmod{p},
$$
using Theorem \thm{newxithm}.
Thus 
$$
(1-q)^s F((1-q)^r, pn - 1) 
\equiv \sum_{i\in S(p) \setminus \{i_0\}} (1-q)^{ri+s} A_p(pn-1,i,(1-q^p)^{r}) 
+ O(q^{pn}) \pmod{p}.
$$
By \eqn{binomcong} we see that the only terms $q^{j'}$ that occur in 
$(1-q)^{ri+s}$
(where $i\in S(p) \setminus \{i_0\}$) satisfy
$$
j' \equiv j \pmod{p} \quad \text{where $0\le j \le m$ and $m\in S^{*}(p,r,s)$.}
$$
This is because  $i\in S(p) \setminus \{i_0\}$ if and only if $ri+s$ is 
congruent  to an element of $S^{*}(p,r,s)$.
Since $A_p(pn-1,i,(1-q^p)^r)$ is a polynomial in $q^p$ the result follows by letting
$n\to\infty$; i.e.\ every term $q^{j'}$ in 
$$
(1-q)^s F( (1-q)^r ) 
= \sum_{n=0}^\infty \sum_{j=0}^s \binom{s}{j} (-1)^j \xi_r(n-j) q^n
$$ 
where $j'$ is congruent
to an element of $T^{*}(p,r,s)$ must have a coefficient that is congruent to $0$
mod $p$.

\subsection*{Case II} $r<0$. This time we choose integers $\beta$ and $m$ such that
$$
r = mp + \beta,
$$
where $0 < \beta \le p-1$ and $m < 0$. We find that
$$
(1 - q)^r =  (1 - q)^\beta \Phi(q^p) + p \Psi(q) 
\equiv (1 - q)^\beta \Phi(q^p) \pmod{p},       
$$
where $\Phi(q)$, $\Psi(q)\in \mathbb{Z}[[q]]$ and $\Phi(q)$ has constant term $1$ so that
it is a unit in the ring of formal power series. In fact,
$$
\Phi(q) = (1 - q)^{m} = 1 + \sum_{k=1}^\infty \binom{k-m-1}{k} q^k = 1 - mq + \frac{m(m-1)}{2} q^2 + \dots.
$$
We basically proceed as in Case I. We have
\begin{align*}
&F((1-q)^r, pn - 1) = 
F((1 - q)^\beta \Phi(q^p) + p \Psi(q), pn-1)\\
&= \sum_{i=0}^{p-1} ((1 - q)^\beta \Phi(q^p) + p \Psi(q))^i 
A_p(pn-1,i,((1 - q)^\beta \Phi(q^p) + p \Psi(q))^p) \\
&\equiv \sum_{i=0}^{p-1} (1 - q)^{\beta i} \left[\Phi(q^p)\right]^i  
A_p(pn-1,i,((1 - q^p)^\beta \left[\Phi(q^p)\right]^p ) \pmod{p}\\
&\equiv \sum_{i\in S(p) \setminus \{i_0\}} (1-q)^{\beta i} \left[\Phi(q^p)\right]^p 
A_p(pn-1,i,((1 - q^p)^\beta \left[\Phi(q^p)\right]^p ) 
+ O(q^{pn}) \pmod{p},
\end{align*}
and
\begin{align*}
&(1-q)^s F((1-q)^r, pn - 1) \\ 
&\equiv \sum_{i\in S(p) \setminus \{i_0\}} (1-q)^{\beta i + s} \left[\Phi(q^p)\right]^p 
A_p(pn-1,i,((1 - q^p)^\beta \left[\Phi(q^p)\right]^p ) 
+ O(q^{pn}) \pmod{p},
\end{align*}
arguing as before. This time
instead of the term
$$
A_p(pn-1,i,(1-q^p)^r),
$$ 
which is a polynomial in $q^p$ with integer coefficients we have the term
$$
\left[\Phi(q^p)\right]^p 
A_p(pn-1,i,((1 - q^p)^\beta \left[\Phi(q^p)\right]^p ),
$$
which is a formal power series in $q^p$ with integer coefficients.
The result follows as before by letting $n\to\infty$.
This completes the proof of our main theorem.

We now prove Corollary \corol{congrels}.
As before suppose $p\ge 5$ is prime and let $r$ be a fixed nonzero integer
relatively prime to $p$. Suppose $\overline{r}$ is the multiplicative
inverse of $r$ mod $p$. We see that
$$
r\tfrac{1}{2}n(3n-1) \equiv -1-s \pmod{p} \quad\text{if and only if}\quad
(6n-1)^2 \equiv -24(1+s)\overline{r} + 1 \pmod{p}.
$$
Thus $p-1\not\in S^{*}(p,r,s)$ if $-24(1+s)\overline{r}+1$ is
either a quadratic nonresidue mod $p$ or congruent to zero mod $p$.
There are $\tfrac{1}{2}(p+1)$ such values of $s$ for $0 \le s \le p-1$.
Thus
\beq
\sum_{j=0}^s \binom{s}{j} (-1)^j \xi_r(pn+p-1 -j) \equiv 0 \pmod{p},
\mylabel{eq:nicecongrel}
\eeq
for all $n\ge0$ provided $\leg{-24(1+s)\overline{r}+1}{p}=-1$ or $0$.
It is clear that this set of $\tfrac{1}{2}(p+1)$ congruence relations
mod $p$ is linearly independent. This completes the proof of
Corollary \corol{congrels}.

We illustrate \eqn{nicecongrel} with some examples.
\begin{example}
$s=0$ We have
$$
\xi(pn + p - 1) \equiv 0 \pmod{p}
$$
for all $n\ge 0$ provided $p\ge5$ is prime and $\leg{-23}{p}=0$ or $-1$; i.e.
$p\equiv 0, 5, 7, 10, 11, 14, 15, 17$, $19$, $20, 21$, or $22 \pmod{23}$.
\end{example}

\begin{example}
$s=1$ We have
$$
\xi(pn + p - 1) \equiv \xi(pn + p -2) \pmod{p}
$$
for all $n\ge 0$ provided $p\ge5$ is prime and $\leg{-47}{p}=0$ or $-1$; i.e.
$p\equiv 0, 5, 10, 11, 13, 15, 19, 20$, $22$, 
$23, 26, 29, 30, 31, 33, 35, 38, 39, 40, 41, 43, 44, 45$, or $46 \pmod{47}$.
\end{example}

\begin{example}
$s=2$ We have
$$
\xi(pn + p - 1) - 2\,\xi(pn + p -2) + \xi(pn + p -3) \equiv 0\pmod{p}
$$
for all $n\ge 0$ provided $p\ge5$ is prime and $\leg{-71}{p}=0$ or $-1$; i.e.
$p=71$ or $\leg{71}{p}=-1$. We have
$$
\xi_{-1}(pn + p - 1) - 2\,\xi_{-1}(pn + p -2) + \xi_{-1}(pn + p -3) \equiv 0\pmod{p}
$$
for all $n\ge 0$ provided $p\ge5$ is prime and $\leg{73}{p}=0$ or $-1$; i.e.
$p=73$ or $\leg{73}{p}=-1$. 
\end{example}

\section{Conclusion}
\mylabel{sec:theend}

We pose the following problems.

\begin{enumerate}
\item[(i)]
The numbers $\xi(n)$ and $\xi_{-1}(n)$ have many combinatorial
interpretations \cite{Br-Li-Rh}, \cite{Kh}, \cite{Le}.
Use one of the interpretations to find a rank or crank-type function
\cite{An-Ga88} to explain combinatorially the simplest congruences

\begin{align*}
\xi(5n+3) \equiv \xi(5n+4) &\equiv 0 \pmod{5},\\
\xi_{-1}(5n+4) &\equiv 0 \pmod{5}.
\end{align*}
\item[(ii)]
Zagier \cite{Za} showed that $q^{1/24}F(q)$ is a so-called quantum
modular form \cite{Za10}. Find and prove congruences for the
coefficients of other quantum modular forms. In particular
look at the functions considered by Andrews, Jim\'enez-Urroz and
Ono \cite{An-JU-On}.
\end{enumerate}

\bigskip

\noindent
\textbf{Acknowledgements}

\noindent
I would like to thank George Andrews, Sharon Garthwaite, Ira Gessel, 
Ken Ono and Rob Rhoades 
for their comments and suggestions. In particular I  thank
Ira Gessel for reminding me of the 
exponential generating function for the signless
Stirling numbers of the first kind, and I thank Sharon Garthwaite
and Rob Rhoades for sharing their observations on congruences
for pairs and triples for the sequences $\xi(n)$ and $\xi_{-1}(n)$.
\bibliographystyle{amsplain}

\end{document}